\documentclass{amsart}

\usepackage{amsthm, amsfonts, amssymb, amsmath, latexsym, enumerate, times}
\usepackage{amscd}
\usepackage{xypic}  %commu
\usepackage{amssymb}
\usepackage{amsthm}
\usepackage{epsfig}
\usepackage{paralist}
\usepackage{lmodern}
\usepackage[T1]{fontenc}
\usepackage[utf8]{inputenc}
\usepackage{mathrsfs}
\usepackage{array}
\usepackage{comment}
\usepackage{paralist}
\usepackage{color}
\newtheorem{theorem}{Theorem}[section]

\newtheorem{claim}[theorem]{Claim}

\theoremstyle{definition}

\newcommand{\Pic}{{\rm Pic}}

\newcommand{\p}{{\mathbb P}}

\def\leq{\leqslant}

\begin{document}
 
\title[Footnotes to a footnote of a paper by Fano]{Footnotes to a footnote of a paper by Fano}

\author{Ciro Ciliberto}
\address{Dipartimento di Matematica, Universit\`a di Roma Tor Vergata, Via O. Raimondo 00173, Roma, Italia}
\email{cilibert@mat.uniroma2.it}

\subjclass{Primary ; Secondary }
 
 
\maketitle

\medskip

\begin{abstract} Following a suggestion in a footnote of the paper \cite {Fa} by Fano, we give a direct proof of the fact that  one of the two Fano threefolds of maximal degree 72 in $\p^{38}$ is isomorphic to the anticanonically embedded weighted projective space $\p(1,1,4,6)$.  \end{abstract}

\section*{Introduction} 

In the paper \cite {Fa} Fano claimed that anticanonically embedded Fano threefolds have degree at most  72, and that there are only two such threefolds of degree 72, that he fully described in a geometric way. The proof of Fano's theorem is certainly not sufficient from the viewpoint of the modern standards of rigor, and a modern and fully rigorous proof has been given by Prokhorov in 
\cite {Pro1}. Prokhorov proved in \cite {Pro1} that if $X$ is a Fano threefold with canonical Gorenstein singularities, then $-K_X^3\leq 72$ and the equality is attained if and only if $X$ is either the weighted projective space $\p(1,1,1,3)$ or the weighted projective space $\p(1,1,4,6)$.  Prokhorov  identifies the two Fano varieties  $X$ such that $-K_X^3=72$. One of them is easily seen to be $\p(1,1,1,3)$ anticanonically embedded. As for the other one,  it is obtained from a very nice geometric construction by Fano (recalled here in \S \ref  {sec:FC1}). To prove that this variety coincides with $\p(1,1,4,6)$ Prokhorov basically argues as follows. The weighted projective space $\p:=\p(1,1,4,6)$ has $|-K_\p|$ very ample, has  $-K_{\p}^3=72$ and $\dim (|-K_{\p}|)=38$, so $\p$ anticanonically embeds in $\p^{38}$ as a threefold of degree 72. By exclusion, the threefold constructed by Fano  must be $\p(1,1,4,6)$. This argument is perfectly all right, however it would be desirable to have also a somewhat explicit proof that the variety arising from Fano's contruction coincides with $\p(1,1,4,6)$ anticanonically embedded. This is what we do in this note (see Theorem \ref {thm:main}). Our argument is contained in \S \ref {sec:proof} and it is based on giving explicit parametric equations for the variety arising from Fano construction. In doing so, we follow a suggestion given by Fano in a footnote of \cite {Fa}.

\medskip
{\bf Acknowledgements:} The author is a member of GNSAGA of the Istituto Nazionale di Alta Matematica ``F. Severi''. The author wants to thank Yuri Prokhorov for  useful e--mail exchanges of ideas on the topic of this note. 

\section{The weighted projective space $\p(1,1,4,6)$} \label{sec:FC}

We consider the weighted projective space $\p(1,1,4,6)$ and we will denote it by $\p:=\p(1,1,4,6)$. It is well known (see \cite{F}) that:\\
\begin{inparaenum}
\item [(i)] $-K_{\p}=\mathcal O_\p(12)$ is a very ample invertible sheaf;\\
\item [(ii)] $\Pic(\p)=\mathbb Z[\mathcal O_\p(12)]$;\\
\item [(iii)] $-K_{\p}^3=72$;\\
\item [(iv)] $\dim (|-K_{\p}|)=38$;\\
\item [(v)] the anticanonical image of $\p$ (of degree 72 in $\p^{38}$) is a Fano threefold.
\end{inparaenum}

\section{Fano's construction} \label{sec:FC1}

Let us recall Fano's construction  from \cite [\S 12, f)]{Fa} (see also \cite [Chapt. IV, 4.2(ii)]{Isk}). Consider the minimal rational surface $\mathbb F_4$,  where we denote by $E$ the negative section with $E^2=-4$ and by $F$ 
the ruling with $F\cdot E=1$ and $F^2=0$. The linear system $|E+6F|$ is very ample on $\mathbb F_4$ and determines an embedding of $\mathbb F_4$ in $\p^9$ whose image is a scroll surface $Y$ of minimal degree 8. The section $E$ is mapped to a conic and the rulings $F$ to lines. Note that $Y$ is the isomorphic image of
$$
\p(\mathcal O_{\p^1} (2)\oplus \mathcal O_{\p^1} (6))
$$
via the morphism determined by its (very ample) tautological line bundle. 

Let $X$ be the cone over $Y$ in $\p^{10}$ with vertex a point $p$. Then $X$ is the image of 
$$
\widetilde X:=\p(\mathcal O_{\p^1}\oplus \mathcal O_{\p^1} (2)\oplus \mathcal O_{\p^1} (6))
$$
via the morphism determined by its (base point free)  tautological line bundle $L$. The smooth threefold $\widetilde X$ is a $\p^2$--bundle over $\p^1$, whose general fibre we will denote by $P$. 
%In the map $\varphi: \widetilde X\longrightarrow X$, the section $C$ corresponding to the quotient 
%$$\mathcal O_{\p^1}\oplus \mathcal O_{\p^1} (2)\oplus \mathcal O_{\p^1} (6)\longrightarrow \mathcal O_{\p^1}$$
%is contracted to the vertex $p$ of $X$ that is a point of multiplicity 8. 
The planes $P$ are mapped to the joins of the vertex $p$ with the lines of the ruling of the scroll $Y$. 

%Moreover $X$ contains a quadric cone $Q$ that is the cone over $E$ from $p$, whose counterimage via $\varphi$ is a surface $\widetilde Q$ isomorphic to $\mathbb F_2$ that is the resolution of $Q$, with as exceptional curve the curve $C$. Note that on $\widetilde X$ we have
%$$\widetilde Q\sim L-6P.$$

What Fano does in \cite [\S 12, f)]{Fa} is to consider the linear system $\mathcal L$ of Weil divisors on $X$ cut out by the cubic hypersurfaces containing 6 planes of the ruling of $X$. Then he considers the rational map $\phi_\mathcal L: X\dasharrow Z\subset \p^{n}$, where $n=\dim(\mathcal L)$, determined by the linear system $\mathcal L$, whose image we denote by $Z$.

Of course this is essentially equivalent to consider on $\widetilde X$ the linear system $\widetilde {\mathcal L}:=|3L-6P|$ and the map $\phi_{\widetilde {\mathcal L}}: \widetilde X\dasharrow Z\subset \p^{n}$. It is an easy exercise to compute $n=\dim(\widetilde {\mathcal L})=\dim(\mathcal L)=38$. 

The main result in this note is the following:

\begin{theorem}\label{thm:main} The image $Z$ of the map $\phi_\mathcal L: X\dasharrow Z\subset \p^{38}$ coincides with the anticanonical image of $\p(1,1,4,6)$.
\end{theorem}

\section{The proof of Theorem \ref {thm:main} } \label{sec:proof}

\begin{proof}[Proof of Theorem \ref {thm:main}] Let $P_1,\ldots, P_5$ be general planes of the ruling of $X$. Fix general lines $\ell_1,\ell_2$ in $P_1,P_2$ respectively and general points $p_i\in P_i$, for $3\leq i\leq 5$. Then project down $\psi: X\dasharrow \p^3$ from $\langle \ell_1, \ell_2, p_3,p_4,p_5\rangle$. It is immediate to see that $\psi: X\dasharrow \p^3$ is a birational map that maps the ruling of planes of $X$ to the pencil $\mathcal P$ of planes in $\p^3$ passing through a given line $r$, and the vertex $p$ of $X$ is mapped to a point $q\in r$. We may choose homogeneous coordinates $[x_1,x_2, x_3,x_4]$ in $\p^3$ so that the images $\alpha_1,\alpha_2$ of the planes $P_1,P_2$ have equations $x_1=0$, $x_2=0$ respectively (hence $r$ has equations $x_1=x_2=0$), and the images $\alpha_3,\alpha_4,\alpha_5$ of the three planes $P_3,P_4,P_5$ have global equation $\xi(x_1,x_2)=0$, where $\xi(x_1,x_2)$ is a suitable homogeneous polynomial of degree 3 in $x_1,x_2$. We may also assume that the point $q$, image of the vertex $p$ of $X$, coincides with the point with coordinates $[0,0,0,1]$.

The linear system $|\mathcal O_X(1)|$ is mapped by $\psi: X\dasharrow \p^3$  to a linear system $\mathcal S$ of dimension 10 of scroll surfaces, that we want to describe. First of all, since a general hyperplane section of $X$ (a scroll surface of degree 8 isomorphic to $Y$) intersects the lines $\ell_1,\ell_2$ at two points, it is clear that the surfaces in $\mathcal S$ have degree 6. Moreover the variable intersection of the general surface in $\mathcal S$ with a plane of the pencil $\mathcal P$ is a line. This implies that the surfaces in $\mathcal S$ have multiplicity 5 along $r$. 

\begin{claim} \label{claim:one} The surfaces in $\mathcal S$ intersect the planes $\alpha_1,\alpha_2$, off $r$, in lines passing through the point $q$.
\end{claim}

\begin{proof}[Proof of the Claim \ref {claim:one}]  We argue for the plane $\alpha_1$, the argument is the same for $\alpha_2$. Let us see what the projection $\psi: X\dasharrow \p^3$ does to the plane $P_1$. First of all we have to blow--up the line $\ell_1$ in $X$. This produces a threefold $X'\longrightarrow X$, with an exceptional divisor $\mathcal E$ over $\ell_1$, that is isomorphic to $\mathbb F_1$ because the normal bundle of $\ell_1$ in $X$ is clearly isomorphic to $\mathcal O_{\p^1}\oplus \mathcal O_{\p^1}(1)$. The projection $\psi$ induces a birational map $\psi': X'\dasharrow \p^3$. 
Via $\psi'$ the strict transform on $X'$ of the plane $P_1$ is contracted to a point, and precisely to the point   $q$, whereas the image of $\mathcal E$ is the plane $\alpha_1$ (via $\psi'$  the $(-1)$--section of $\mathcal E$, that is the intersection of $\mathcal E$ with the strict  transform on $X'$ of the plane $P_1$, gets contracted to the point $q$). Let $S$ be a general hyperplane section of $X$, let $s$ be the intersection line of $S$ with $P_1$, and let $x$ be the intersection of $s$ with $\ell_1$. Via $\psi'$  the surface  $S$ undergoes an elementary transformation based at $x$. In this elementary transformation $s$ is contracted to the point $q$ and the exceptional divisor of the blow--up of $S$ at $X$, that sits on $\mathcal E$, is mapped to the ruling of the surface $\Sigma$ image of $S$ via $\psi'$, contained in $\alpha_1$ and it passes through $q$, as wanted. \end{proof}

\begin{claim}\label{claim:two} The surfaces in $\mathcal S$ intersect the planes $\alpha_3,\alpha_4,\alpha_5$ in the line $r$ with multiplicity 6. 
\end{claim}

\begin{proof} [Proof of the Claim \ref {claim:two}] We argue for the plane $\alpha_3$, the argument is the same for $\alpha_4$ and $\alpha_5$. Again, let us see what the projection $\psi: X\dasharrow \p^3$ does to the plane $P_3$. First of all we have to blow--up the point $p_3$ in $X$. This produces a threefold $X'\longrightarrow X$, with an exceptional divisor $\mathcal E\cong \p^2$ over $p_3$. The projection $\psi$ induces a birational map $\psi': X'\dasharrow \p^3$. 
Via $\psi'$ the strict transform on $X'$ of the plane $P_1$ is mapped to a line, and precisely to the line $r$, whereas the image of $\mathcal E$ is the plane $\alpha_3$. Let $S$ be a general hyperplane section of $X$, let $s$ be the intersection line of $S$ with $P_3$. Via $\psi'$  the line $s$ is clearly mapped to $r$, and the assertion follows. \end{proof}

Let us now consider the linear system $\mathcal S'$ of surfaces of degree 6 in $\p^3$ with the line $r$ with multiplicity 5, intersecting the planes $\alpha_1,\alpha_2$, off $r$, in lines passing through the point $q$ and intersecting the planes $\alpha_3,\alpha_4,\alpha_5$ in the line $r$ with multiplicity 6. Of course $\mathcal S\subseteq \mathcal S'$. 

\begin{claim}\label{claim:three} $\mathcal S=\mathcal S'$. 
\end{claim}

\begin{proof} [Proof of the Claim \ref {claim:three}] One has $\dim(\mathcal S)=10$. So it suffices to prove that $\dim(\mathcal S')\leq 10$. To see this, note that we can split off $\mathcal S'$ the planes $\alpha_3,\alpha_4,\alpha_5$ with at most one condition each, and the planes $\alpha_1,\alpha_2$ with at most two conditions each, for a total of 7 conditions at most, and the residual system is the linear system of all planes in $\p^3$, that has dimension 3. The assertion follows. 
\end{proof}

Now we can give an explicit equation for the linear system $\mathcal S$. Indeed $\mathcal S$ has equation of the form
\begin{equation}\label{eq:par}
ax_1x_2x_4\xi(x_1,x_2)+x_3\xi(x_1,x_2)\varphi_2(x_1,x_2)+\varphi_6(x_1,x_2)=0
\end{equation}
where $a$ is a complex parameter and $\varphi_2(x_1,x_2), \varphi_6(x_1,x_2)$ are arbitrary homogeneous polynomials in $x_1,x_2$ of degree equal to the index. Indeed, the equation \eqref {eq:par} depends linearly on 11 independent homogeneous parameters, i.e., $a$ and the coefficients of $\varphi_2(x_1,x_2)$ and $ \varphi_6(x_1,x_2)$, hence it is the equation of a linear system of dimension 10 of surfaces of degree 6. Let us intersect with a plane with equation $x_2=tx_1$, with $t\in \mathbb C$. The resulting system is 
$$
x_2=tx_1, \quad x_1^5(ax_4t\xi(1,t)+x_3\xi(1,t)\varphi_2(1,t)+x_1\varphi_6(1,t))=0.
$$
The factor $x_1^5$ stays for the line $r$ which splits off from the intersection with multiplicity 5. The further intersection has equation 
$$
x_2=tx_1, \quad ax_4t\xi(1,t)+x_3\xi(1,t)\varphi_2(1,t)+x_1\varphi_6(1,t)=0.
$$
If we set $t=0$ (i.e., if we intersect with the plane $x_2=0$, namely $\alpha_2$) we get the equations
$$
x_2=0, \quad x_3\xi(1,0)\varphi_2(1,0)+x_1\varphi_6(1,0)=0
$$
that is the equation of a line through $q$. The same if we intersect with the plane $\alpha_1$. If we intersect with either one of the planes $\alpha_3,\alpha_4,\alpha_5$, this means that we choose a value $t=\tau$ such that $\xi(1,\tau)=0$. Hence the resulting equations are
$$
x_2=\tau x_1, \quad x_1\varphi_6(1,\tau)=0 \Longleftrightarrow x_1=x_2=0,
$$
namely we get the line $r$ once more. This shows that \eqref {eq:par} is just the equation of the linear system $\mathcal S$. 

The linear system $\mathcal S$ determines a map $\phi_\mathcal S: \p^3\dasharrow X\subset \p^{10}$, and we can consider the composite map $\phi_\mathcal L\circ \phi_\mathcal S: \p^3\dasharrow Z\subset \p^{38}$, that in turn corresponds to a linear system $\mathcal T$ of dimension 38 of surfaces in $\p^3$. Let us identify the system $\mathcal T$. 

Remember that $\mathcal L$ is cut out  on $X$ by the cubic hypersurfaces containing 6 planes of the ruling of $X$. This implies that $\mathcal T$ has to be a linear system of surfaces of degree 12, that pass through the line $r$ with multiplicity $9$. Specifically, we note that the (reducible) surface with equation
$$
x_1^2x_2^2x_4^2\xi(x_1,x_2)^2=0
$$
belongs to $\mathcal T$. Indeed, by \eqref {eq:par}, the surface
$$
x_1x_2x_4\xi(x_1,x_2)\cdot x_1x_2x_4\xi(x_1,x_2)\cdot \varphi_6(x_1,x_2)=0
$$
is the image on $\p^3$ of a surface cut out on $X$ by a cubic hypersurface. Dividing by $\varphi_6(x_1,x_2)$ we get a surface in $\mathcal T$. 

Arguing in a similar manner, we see that the surfaces with the following equations are in $\mathcal T$
$$
\begin{aligned}
&x_1x_2x_3x_4\xi(x_1,x_2)^2f_2(x_1,x_2)=0, \,x_1x_2x_4\xi(x_1,x_2)f_6(x_1,x_2)=0,\, x_3^3\xi(x_1,x_2)^3=0, \cr
&x_3^2\xi(x_1,x_2)^2f_4(x_1,x_2)=0, \quad x_3\xi(x_1,x_2)f_8(x_1,x_2)=0, \quad f_{12}(x_1,x_2)=0,
\end{aligned}
$$
where $f_i(x_1,x_2)$ are homogeneous polynomials of degree $i$ in $x_1,x_2$, for $i\in \{2,4,6,8,\\
12\}$. Hence the following is the equation of a sublinear system of $\mathcal T$
\begin{equation}\label{eq:wort}
\begin{aligned}
&ax_1^2x_2^2x_4^2\xi^2(x_1,x_2)+x_1x_2x_4\xi(x_1,x_2)(x_3\xi(x_1,x_2)f_2(x_1,x_2)+f_6(x_1,x_2))+\cr
&+bx_3^3\xi^3(x_1,x_2)+x_3^2\xi^2(x_1,x_2)f_4(x_1,x_2)+x_3\xi(x_1,x_2)f_8(x_1,x_2)+f_{12}(x_1,x_2)=0
\end{aligned}
\end{equation}
where $a,b$ are complex parameters and $f_i(x_1,x_2)$, with $i\in \{2,4,6,8,12\}$, are arbitrary homogeneous polynomials of degree equal to the index, in $x_1,x_2$. The polynomial in \eqref {eq:wort} depends linearly on 39 independent homogeneous parameters, hence it defines a linear system of dimension 38 and since $\dim(\mathcal T)=38$, the equation \eqref {eq:wort} defines just the linear system $\mathcal T$.

Finally we are in position to finish our proof. Indeed, consider  the following birational morphism
$$
\begin{aligned}
&\eta: [x_1,x_2,x_3,x_4]\in \p^3\mapsto \cr
&\mapsto [y_1,y_2,y_3,y_4]=[x_1,x_2, x_3\xi(x_1,x_2), x_1x_2x_4\xi(x_1,x_2)]\in \p(1,1,4,6).
\end{aligned}
$$

As we saw in \S \ref {sec:FC}, the anticanonical system on $\p=\p(1,1,4,6)$ is $\mathcal O(12)$, hence the anticanonical morphism $\phi_{|-K_\p|}: \p\longrightarrow \p^{38}$ is determined by the linear system of surfaces of $\p(1,1,4,6)$ with equation
$$
a y_4^2+ y_3y_4f_2(x_1,x_2)+y_4f_6(x_1,x_2)+b y_3^3+y_3^2f_4(x_1,x_2)+y_3f_8(x_1,x_2)+f_{12}(x_1,x_2)=0
$$ 
with $a,b$ and the $f_i$'s as above. Then it is clear from the equation \eqref {eq:wort} of $\mathcal T$ that $\phi_{\mathcal L}= \phi_{|-K_\p|} \circ \eta$, and from this the assertion follows. 
\end{proof}

\end{document}